\documentclass[12pt,reqno]{amsart}
\usepackage{amssymb}
\usepackage{amsmath, mathtools}
\usepackage{amsthm}
\usepackage{amscd}

\newcommand{\RNum}[1]{\uppercase\expandafter{\romannumeral #1\relax}}

\usepackage[T2A]{fontenc}
\usepackage[utf8]{inputenc}
\usepackage[english]{babel}

\input{int.def} 

\usepackage[sort]{cite}
\usepackage{tikz-cd}
\usetikzlibrary{cd}
\usepackage{dirtytalk}
\usepackage[linktoc=page, colorlinks, linkcolor=blue, citecolor=blue]{hyperref}

\usepackage{xcolor}
\usepackage{centernot}

\usepackage{enumitem}

\usepackage{pgfplots}
\usepackage{multicol}

\pgfplotsset{compat=1.17}
\numberwithin{equation}{section}

\DeclarePairedDelimiter \abs{\lvert}{\rvert} 
\DeclarePairedDelimiter \norm{\lVert}{\rVert}
\DeclarePairedDelimiterX \ip[2]{\langle}{\rangle}{#1,#2}
\DeclarePairedDelimiterXPP \Prob[1]{\mathbb{P}}\{\}{}{\newcommand\given{\nonscript\:\delimsize\vert\nonscript\:\mathopen{}} #1} 
\DeclarePairedDelimiterXPP \Probevent[1]{\mathbb{P}}(){}{\newcommand\given{\nonscript\:\delimsize\vert\nonscript\:\mathopen{}}#1} 

\DeclareMathOperator{\Bin}{Bin}
\DeclareMathOperator{\disc}{disc}
\DeclareMathOperator{\E}{\mathbb{E}}

\newcommand{\one}{\mathbf{1}}

\def \R {\mathbb{R}}
\def \e {\varepsilon}

\usepackage[mathcal]{euscript}

\usepackage{geometry}
\newgeometry{vmargin={25mm}, hmargin={22mm,22mm}, footskip=10mm}   

\begin{document}
\title{Random sets are close to low-discrepancy sets}

\thanks{R.V. acknowledges support from NSF Grant DMS-2451011 and U.S. Air Force Grant FA9550-25-1-0294.}

\author{Gleb Smirnov}
\address{
Mathematical Sciences Institute, 
Australian National University, Canberra, Australia}
\email{gleb.smirnov@anu.edu.au}

\author{Roman Vershynin}
\address{Department of Mathematics, University of California, Irvine, US}
\curraddr{}
\email{rvershyn@uci.edu}

\thanks{}


\begin{abstract}
We show that a random sample from an 
arbitrary probability measure on $\R^d$ is close to a low-discrepancy point set. Namely, after moving only a small fraction of the sample points in expectation, one obtains an $n$-point set with star discrepancy \(\operatorname{polylog}(n)/n\) with respect to the original measure.
\end{abstract}

\maketitle

\section{Main results}

A fundamental limitation of random sampling is that its typical estimation error is of order \(n^{-1/2}\). We ask whether this barrier can be overcome by modifying only a small fraction of the sample, assuming the underlying distribution is known. 

\subsection{Empirical processes and star discrepancy}

Let \(X_1,\ldots,X_n\) be iid random variables taking values in \(\R^d\). Let
\(\mu\) denote their common distribution, and let \(\mu_n^X\) denote the empirical
measure:
\[
\mu(B)=\Prob{X_i\in B},
\qquad
\mu_n^X(B)=\frac1n\sum_{i=1}^n\one_{\{X_i\in B\}},
\]
for every Borel set \(B\subset\R^d\).

The classical VC theory (see e.g. \cite[Theorem~8.3.15]{vershynin2026high}) implies that the empirical measure approximates the population measure uniformly over lower orthants:
\begin{equation}\label{eq:VC}
\E \sup_B \abs*{\mu_n^X(B)-\mu(B)}
\le
C\sqrt{\frac dn}.
\end{equation}
Here the supremum is over all lower orthants
\begin{equation}	\label{eq: lower orthant}
	B=(-\infty,t_1]\times\cdots\times(-\infty,t_d],
\qquad
(t_1,\ldots,t_d)\in\R^d.
\end{equation}

The quantity \(\sup_B \abs*{\mu_n^X(B)-\mu(B)}\) has two classical interpretations, one in probability and the other in discrete geometry. For random points \(X_i\), it is the supremum of the empirical process indexed by
lower orthants; see
\cite{Stute,B-Kiefer-Ros,Kiefer-W,Revesz,Massart}. For deterministic point sets,
it is precisely the star discrepancy.

It is known that there exist \(n\)-point sets with star discrepancy as low as 
\(\operatorname{polylog}(n)/n\); see
\cite{Mat,Chaz,Beck-Chen,Aist-Dick,Aist-B-Nik}. This is substantially smaller
than the \(n^{-1/2}\) bound in \eqref{eq:VC}, showing that low-discrepancy point
sets are considerably more regular than random samples.

\subsection{Main result}

Our main result shows that a random sample is close to a low-discrepancy point set. By modifying only a prescribed expected number of
sample points, one can reduce its discrepancy to nearly the optimal order.

\begin{theorem}[A random sample is close to a low-discrepancy set]\label{thm_main}
Let \(X_1,\ldots,X_n\) be iid random variables taking values in \(\R^d\), and let
\(\mu\) denote their common distribution. For every \(m\in(0,n]\), there exists a
randomized algorithm, possibly depending on \(\mu\), that maps points
\(X_1,\ldots,X_n\) to points \(Y_1,\ldots,Y_n\in\R^d\) such that
\[
\E\,\#\{i:Y_i\neq X_i\}\le m,
\]
and
\[
\E\sup_B\abs*{\mu_n^Y(B)-\mu(B)}
\le \frac{C^d \log^{3d/2+1}(2n)}{m}.
\]
Here \(\mu_n^Y\) is the empirical measure of \(Y_1,\ldots,Y_n\), the supremum is
over all lower orthants \(B\), and \(C_d>0\) depends only on \(d\).
The expectations are taken over both the sample and the internal randomness of
the algorithm.
\end{theorem}

In particular, if we choose $m=0.01n$, Theorem~\ref{thm_main} implies that moving just $1\%$ of the sample points can reduce the star discrepancy to $\operatorname{polylog}(n)/n$.

Although the exponent of the logarithm in Theorem~\ref{thm_main} is certainly not optimal, the rate $1/m$ is optimal even in dimension $d=1$ \cite{Sm-Versh-1}.

\subsection{Prior work}
    Several partial cases of Theorem~\ref{thm_main} are known. The case $m=n$ corresponds to finding any $n$-point set in $\R^d$ with small star discrepancy, and the best known constructions give $O_d (\log^{d-1/2}(2n)/n)$ \cite{Aist-B-Nik}.
    In dimension \(d=1\), Theorem~\ref{thm_main} was proved in \cite{Sm-Versh-1} with the optimal bound $O(1/m)$. Moreover, still in dimension $d=1$, it is known how to {\em remove} (rather than move) $m$ points to achieve discrepancy $O(\log(2n)/m)$, which is optimal if only removals are allowed \cite{Bilyk-St}. Perhaps the closest result to ours is \cite{DFGGR}, which proves a version of Theorem~\ref{thm_main} in the particular case of the uniform measure on the unit cube $[0,1]^d$. They show how to remove $m$ points to achieve star discrepancy $O(\log^{2d}(2n)/m)$. The proof in \cite{DFGGR} relies on the orthogonality of Haar wavelets, and is therefore specific to the product structure of the measure $\mu$. Our argument is completely different and it applies to arbitrary Borel measures.

\subsection{Sparse families}

The star discrepancy measures the uniform approximation over a specific family of sets -- the lower orthants in $\R^d$. Let's now be more ambitious and try to consider general families of subsets. 

We will now state an abstract result, from which Theorem~\ref{thm_main} will be deduced in \S\,\ref{proof_thm_main} by standard tricks -- VC theory, dyadic decomposition, and approximation.

A family \(A_1,\ldots,A_D\) of subsets of $[N]=\{1,\ldots,N\}$ is called \emph{\(s\)-sparse} if each point is covered by at most $s$ subsets, i.e.
\[ 
\#\{j\,:\, i \in A_j\}\le s \quad\text{for each } i \in [N]. 
\]

\begin{theorem}[Sparse families]\label{thm_sparse}
Fix \(0 < \delta < 1\). Let \(A_1,\ldots,A_D\) be an \(s\)-sparse family of subsets of $[N]$. Assume that $N$ has the form $N=n2^k$ for some positive integers $n$ and $k$. Let \(S \subset [N]\) be a random \(n\)-element subset, chosen uniformly without replacement. Then there is a randomized algorithm that produces an \(n\)-element subset \(Q \subset [N]\) such that
$$
\E \abs{Q \cap S} \ge n (1 - \delta) - 2 \sqrt{n}
$$
and
$$
\abs*{ \frac{\abs{A_j \cap Q}}{n} - \frac{\abs{A_j}}{N}} 
\le 
C \frac{\sqrt{s}}{n} \left( \frac{k}{\delta} + \log(2n) \right)
\qquad \text{for each } j.
$$
Here \(C > 0\) is an absolute constant. The expectation is taken over both the sample $S$ and the internal randomness of the algorithm.\end{theorem}

\subsection{A word on the proof}
In \S\,\ref{fates} and \S\,\ref{proof_thm_sparse}, we deduce Theorem~\ref{thm_sparse} by combining our earlier work on two-color discrepancy \cite{Versh-Sm-3, Sm-Versh-2} with some randomization of Beck's transference principle \cite{Beck, Aist-B-Nik}. The randomized transference principle is developed in the proof of Proposition~\ref{summary_vectors}; we believe this simple argument can be useful in many other applications. 

Then, in \S\,\ref{proof_thm_main}, we deduce Theorem~\ref{thm_main} by approximation based on VC theory.

\section{Iterated thinning}\label{fates}

The proof of Theorem~\ref{thm_sparse} rests on an iterated thinning procedure. We first record the one-step input.

\begin{proposition}[Partial coloring \cite{Versh-Sm-3}]\label{fisher}
	Let $v_1,\ldots,v_m\in\R^D$ be any vectors, and let $\varepsilon_1,\ldots,\varepsilon_m$ be independent fair signs.  For every $T>0$ there is a randomized algorithm that discards some of the terms $\varepsilon_i v_i$ so that
	\begin{equation}\label{in_cube}
		\max_{r\le m} \norm*{\sum_{i\le r\,\textrm{accepted}}\varepsilon_i v_i}_\infty\le T,
	\end{equation}
	and
	\begin{equation}\label{discard_proba}
		\Prob{v_i\text{ is discarded}}\le \frac{\pi}{T}\norm{v_i}_2
		\qquad \text{for each } i.
	\end{equation}
\end{proposition}

Iterating gives a full-coloring form:

\begin{lemma}[Full coloring]\label{full_coloring}
	Let $v_1,\ldots,v_m\in\R^D$ be any vectors satisfying $\norm{v_i}_2\le 1$, and let $	\varepsilon_1,\ldots,\varepsilon_m$ be independent fair signs.  For every $T>0$ there is a randomized algorithm producing signs $\varepsilon'_1,\ldots,\varepsilon'_m$ such that
	\begin{equation}\label{eq:full-coloring-bound}
		\norm*{\sum_{i=1}^m\varepsilon'_i v_i}_\infty
		\le T+C\log(2m),
	\end{equation}
	and
	\begin{equation}\label{eq:full-coloring-flips}
		\Prob{\varepsilon'_i\ne\varepsilon_i}\le \frac{\pi}{T}
		\qquad \text{for each } i.
	\end{equation}
\end{lemma}

\begin{proof}
Apply Proposition~\ref{fisher}.  Keep the original signs on the accepted vectors.  Color the discarded vectors by iterating the same procedure with threshold $2\pi$.  At each iteration the expected number of discarded vectors is at most half the number entering that iteration, so there is a deterministic choice of the auxiliary signs for which at most half survive.  Summing the contributions over the $O(\log(2m))$ iterations gives $C\log(2m)$.  Together with \eqref{in_cube}, this proves \eqref{eq:full-coloring-bound}.  A sign can change only if the corresponding vector was discarded in the first application, and \eqref{eq:full-coloring-flips} follows from \eqref{discard_proba}.
\end{proof}

We will now convert coloring to sampling; our argument can be seen as a randomized version of Beck's transference principle \cite{Beck, Aist-B-Nik}. 

\begin{proposition}[Iterated thinning]\label{summary_vectors}
Fix $0<\delta<1$.  Let $N=n2^k$ and let $v_1,\ldots,v_N\in\R^D$ satisfy $\|v_i\|_2\le 1$.  Let $S\subset[N]$ be obtained by selecting every index independently with probability $2^{-k}$.  Then there is a randomized algorithm producing a subset $Q \subset [N]$ such that
\begin{equation}\label{many_alpha=1}
	\E \abs{Q \cap S} \ge n(1-\delta)
\end{equation}
and 
\begin{equation}\label{telescope}
	\norm*{\frac1n\sum_{i \in Q} v_i - \frac1N\sum_{i=1}^N v_i}_\infty
	\le \frac{C}{n}\left(\frac{k}{\delta}+\log(2n)\right).	
\end{equation}
\end{proposition}

\begin{proof}
For each vector $v_i$, choose its independent ``fate'' 
\[
\big(\e_i(1),\ldots,\e_i(k)\big) \in \{\pm1\}^k
\]
uniformly at random, and realize
\[
S=\big\{ i: \e_i(1) = \cdots = \e_i(k) = +1\big\}.
\]
Start with $I(0) \coloneqq [N]$. In round $r \in [k]$, apply Lemma~\ref{full_coloring} to the vectors $(v_i)_{i\in I(r-1)}$ with input signs $\e_i(r)$ and threshold
\[
T \coloneqq \frac{2\pi k}{\delta}.
\]
Let $I(r)$ consist of the indices from $I(r-1)$ that receive output sign $\e'_i(r)=+1$. Then, for each round $r$,
\[
\sum_{i\in I(r-1)}v_i-2\sum_{i\in I(r)}v_i
=
-\sum_{i\in I(r-1)}\e'_i(r)v_i.
\]
Hence:
\begin{equation}\label{eq: one round}
	\norm*{\sum_{i\in I(r-1)}v_i-2\sum_{i\in I(r)}v_i}_\infty
	=
	\norm*{\sum_{i\in I(r-1)}\e'_i(r)v_i}_\infty
	\le T+C\log(2N).
\end{equation}
Telescoping gives
\[
\sum_{r=1}^k 2^{r-1}
\left(
	\sum_{i\in I(r-1)}v_i
	-
	2\sum_{i\in I(r)}v_i
\right) = \sum_{i=1}^N v_i
-
2^k\sum_{i\in I(k)}v_i,
\]
so we obtain from \eqref{eq: one round}:
\[
\norm*{\sum_{i=1}^N v_i-2^k\sum_{i\in I(k)}v_i}_\infty \le \sum_{r=1}^k 2^{r-1}\bigl(T+C\log(2N)\bigr) \le  
2^k \bigl(T+C\log(2N)\bigr).
\]
Put 
$$
Q \coloneqq I(k).
$$
Then, dividing by $N$, we conclude:
\[
\norm*{
\frac{1}{N}\sum_{i=1}^N v_i - \frac{1}{n}\sum_{i\in Q}v_i}_\infty \le 
\frac{1}{n}\bigl(T+C\log(2N)\bigr),
\]
and \eqref{telescope} follows after 
adjusting the constant \(C\). 

It remains to prove \eqref{many_alpha=1}. By construction, 
$$
Q \cap S = \big\{ i: \e'_i(r) = \e_i(r) = 1 \textrm{ for all } r \in [k] \big\}.
$$

To bound the expected cardinality of $Q \cap S$, let us address each round $r$ separately. Condition on the outcomes of all past rounds (i.e. any configuration of the input and output signs in the rounds $1,\ldots,r-1$ that leave $i \in I(r-1)$). Then, using Lemma~\ref{full_coloring} and the fairness of the sign $\e_i(r)$, we get
$$
\Prob[\big]{\e'_i(r) \ne \e_i(r) \given \text{past rounds and } \e_i(r)=+1}
\le 2 \, \Prob{\e'_i(r) \ne \e_i(r) \given \text{past rounds}}
\le \frac{2\pi}{T}=\frac{\delta}{k}.
$$
Equivalently,
$$
\Prob[\big]{\e'_i(r) = \e_i(r) \given \text{past rounds and } \e_i(r)=+1}
\ge 1 - \frac{\delta}{k}.
$$
From this, using the fairness of the sign again, we obtain
$$
\Prob[\big]{\e'_i(r) = \e_i(r) = 1 \given \text{past rounds}}
\ge \left(1-\frac{\delta}{k}\right) \cdot \frac12.
$$
Hence, telescoping the conditional probability we conclude that
$$
\Prob[\big]{\e'_i(r) = \e_i(r) = 1 \textrm{ for all } r \in [k]}
\ge \left(1-\frac{\delta}{k}\right)^k \cdot \frac1{2^k}
\ge 2^{-k} (1-\delta).
$$
Now sum over $i \in [N]$ to get
\[
\E \abs{Q \cap S} \ge N2^{-k}(1-\delta)=n(1-\delta). \qedhere
\]
\end{proof}

\section{Sparse families}\label{proof_thm_sparse}

We now pass from Proposition~\ref{summary_vectors} to Theorem~\ref{thm_sparse}. The proof has two parts: first we prove the result in the Bernoulli model, then we use coupling to pass to sampling without replacement.

\begin{lemma}[Sparse families: Bernoulli selection]\label{lem:bernoulli-sparse}
Fix \(0 < \delta < 1\). Let \(A_1,\ldots,A_D\) be an \(s\)-sparse family of subsets of $[N]$. Assume that $N$ has the form $N=n2^k$ for some positive integers $n$ and $k$. Let \(S \subset [N]\) be a random subset obtained by selecting each point independently with probability $2^{-k}$. Then there is a randomized algorithm that produces an $n$-element subset $Q\subset [N]$ such that
\begin{equation}\label{eq:bernoulli-overlap}
\E \abs{Q\cap S} \ge n(1-\delta)-\sqrt n
\end{equation}
and
\begin{equation}\label{eq:bernoulli-discrepancy}
	\abs*{ \frac{\abs{A_j \cap Q}}{n} - \frac{\abs{A_j}}{N}} 
	\le C\frac{\sqrt{s}}{n}\left(\frac{k}{\delta}+\log(2n)\right)
	\qquad \text{for each } j.
\end{equation}
\end{lemma}

\begin{proof}
Adjoin $A_0=[N]$ to the family and associate to each $i \in [N]$ the incidence vector
\[
v_i=\bigl(\one_{\{i\in A_0\}},\ldots,\one_{\{i\in A_D\}}\bigr)\in\R^{D+1}.
\]
Because the original family is $s$-sparse,
\[
\|v_i\|_2\le \sqrt{s+1}\le \sqrt{2s}.
\]
Applying Proposition~\ref{summary_vectors} for the scaled vectors $v_i/\sqrt{2s}$, we obtain a subset $Q_0$ satisfying 
\begin{equation}\label{eq:Q0-discrepancy}
	\abs*{ \frac{\abs{A_j \cap Q_0}}{n} - \frac{\abs{A_j}}{N}} 
	\le C\frac{\sqrt{s}}{n}\left(\frac{k}{\delta}+\log(2n)\right)
	\quad \text{for each } j=0,\ldots,D.
\end{equation}
For $j=0$, this implies
\begin{equation}\label{eq:Q0-size}
\bigl||Q_0|-n\bigr|
\le C\sqrt{s}\left(\frac{k}{\delta}+\log(2n)\right).
\end{equation}
Moreover, \eqref{many_alpha=1} yields
\begin{equation}\label{eq:Q0-overlap}
\E|Q_0\cap S|\ge n(1-\delta).
\end{equation}

Modify $Q_0$ to have cardinality exactly $n$: add arbitrary points when $|Q_0|<n$; when $|Q_0|>n$, remove points outside $S$ first and then, only if necessary, remove points of $S$.  Call the resulting set $Q$.  By \eqref{eq:Q0-size}, this adjustment changes every normalized count by at most the right-hand side of \eqref{eq:Q0-discrepancy}, so \eqref{eq:bernoulli-discrepancy} follows after changing $C$.

The prescribed removal rule gives
\[
|Q\cap S|\ge |Q_0\cap S|-(|S|-n)_+.
\]
Since $|S|\sim\operatorname{Bin}(N,2^{-k})$ has mean $N2^{-k}=n$ and variance at most $n$,
\[
\E(\abs{S}-n)_+\le \E\abs[\big]{\abs{S}-n} \le \sqrt n.
\]
Combining this with \eqref{eq:Q0-overlap} proves \eqref{eq:bernoulli-overlap}.
\end{proof}

\begin{lemma}[Coupling Bernoulli and fixed-size sampling]\label{lem:depoissonization}
Let $S$ be a uniformly distributed $n$-element subset of $[N]$, and let $\xi\sim\operatorname{Bin}(N,2^{-k})$ be independent of $S$.  Construct $S'$ by adding $\xi-n$ uniformly chosen points of $[N]\setminus S$ when $\xi>n$, and by removing $n-\xi$ uniformly chosen points of $S$ when $\xi<n$.  Then $S'$ is a Bernoulli subset with inclusion probability $2^{-k}$, and
\begin{equation}\label{eq:coupling-distance}
\E|S\triangle S'|=\E|\xi-n|\le \sqrt n.
\end{equation}
\end{lemma}

\begin{proof}
Conditional on $|S'|=r$, the construction is invariant under every permutation of $[N]$, hence $S'$ is uniform among the $r$-element subsets.  Since $|S'|=\xi\sim\operatorname{Bin}(N,2^{-k})$, it follows that $S'$ has the Bernoulli law, as claimed. The construction changes exactly $|\xi-n|$ points, and Cauchy--Schwarz gives
\[
\E|\xi-n|\le \sqrt{\operatorname{Var}(\xi)}\le \sqrt n.
\qedhere \] 
\end{proof}

\medskip

\begin{proof}[Proof of Theorem~\ref{thm_sparse}]
Couple the given uniform $n$-element set $S$ with a Bernoulli set $S'$ as in Lemma~\ref{lem:depoissonization}.  Apply Lemma~\ref{lem:bernoulli-sparse} to $S'$ and let $Q$ be the resulting $n$-element set.  The discrepancy estimate is exactly \eqref{eq:bernoulli-discrepancy}.  Furthermore,
\[
|Q\cap S|\ge |Q\cap S'|-|S\triangle S'|.
\]
Taking expectations and using \eqref{eq:bernoulli-overlap} and \eqref{eq:coupling-distance}, we obtain
\[
\E|Q\cap S|
\ge n(1-\delta)-2\sqrt n.
\]
This proves the theorem.
\end{proof}

\section{Proof of Theorem~\ref{thm_main}}\label{proof_thm_main}

We will deduce the main theorem from Theorem~\ref{thm_sparse}. 

Let's start with a few reductions as in \cite{Sm-Versh-2}. First, we can assume that the cumulative distribution functions of all marginals of $\mu$ are strictly increasing. This follows by mixing the original distribution with a Gaussian: sample from the original distribution with probability $p$ and from the standard Gaussian distribution with probability $1-p$; then take $p \to 1$. 

Moreover, it is enough to treat measures on $[0,1]^d$. To see this, consider a fixed increasing homeomorphism from $\R$ onto $(0,1)$, applied coordinatewise.  

Finally, the following observation will allow us to focus on measures with uniform marginals: 

\begin{lemma}[Uniformization] \label{lem:uniformization}
	Let $X$ be a random vector with distribution $\mu$ on $[0,1]^d$. Let $F_k$ denote the cumulative distribution function of the $k$th marginal of $\mu$. Assume all $F_k$ are strictly increasing on $[0,1]$. Then there exists a random vector $\widehat X$ whose marginals are uniform on $[0,1]$ and such that, simultaneously for every anchored box
	\begin{equation}	\label{eq: anchored box}
		B=[0,t_1]\times\cdots\times[0,t_d], 
		\qquad (t_1,\ldots,t_d)\in\R^d,
	\end{equation}
	we have
	\[
	X\in B \Longleftrightarrow \widehat X\in\widehat B,
	\]
	where $\widehat B=[0,F_1(t_1)]\times\cdots\times[0,F_d(t_d)]$.
\end{lemma}

\begin{proof}
	The argument is the same as in \cite[Section~3]{Sm-Versh-2}, and we include it here for completeness. Let $U_1,\ldots,U_d$ be independent $\operatorname{Unif}[0,1]$ random variables, independent of $X$. Define the (randomized) integral transform
	\begin{equation}	\label{eq: integral transform}
		\widehat X(k)
		= U_k \cdot F_k(X(k)) + (1-U_k) \cdot F_k(X(k)^-), 
		\qquad k=1,\ldots,d,
	\end{equation}
    where $F_k(a^-) = \lim_{x \to a^-} F_k(x)$. Then $\widehat X(k)$ is uniform on $[0,1]$ for each coordinate $k$.
	
	If $X \in B$ then $\widehat X\in\widehat B$, since $\widehat X(k) \le F_k(X(k))$ by construction. If $X \not\in B$, then $X(k) > t_k$ for some $k$. Hence, there is $y$ with $t_k < y < X(k)$. Then 
	$$
	F_k(t_k) < F_k(y) \le F_k(X(k)^-) \le \widehat X(k),
	$$
	where the first two bounds follow by strict monotonicity of $F_k$, and the last one follows from \eqref{eq: integral transform} if one replaces $F_k(X(k))$ with the smaller value $F_k(X(k)^-)$. Thus, $\widehat{X} \not\in \widehat{B}$. 
\end{proof}

Now we will prove a preliminary version of Theorem~\ref{thm_main}, where the ground measure $\mu$ is replaced with a very large iid sample from $\mu$. 

\begin{proposition}[Sampling from a large sample]\label{prop_sparse_2}
	Fix $0<\delta<1$. Assume that $N$ has the form $N=n2^k$ for some positive integers $n$ and $k$. Let $Z_1,\ldots,Z_N$ be iid random variables taking values in $[0,1]^d$. Let $S\subset[N]$ be a random $n$-element subset, chosen uniformly without replacement, independently of the sample $Z_i$. Then there is a randomized algorithm that produces an $n$-element subset $Q\subset[N]$ such that
	\begin{equation}\label{eq:prop-overlap}
		\E \abs*{Q \cap S} \ge n(1-\delta)-2\sqrt n,
	\end{equation}
	and
	\begin{equation}\label{eq:prop-discrepancy}
		\E\sup_B \abs*{ \frac1n\sum_{i\in Q}\one_{\{Z_i\in B\}} - \frac1N\sum_{i=1}^N\one_{\{Z_i\in B\}}}
		\le \frac{C^d \log^{3d/2}(2n)}{n}
		\left(\frac{k}{\delta}+\log(2n)\right),
	\end{equation}
where the supremum is over anchored boxes \eqref{eq: anchored box}.
\end{proposition}

\begin{proof}
Let $\mu$ denote the law of $Z_i$. Applying the uniformization \eqref{eq: integral transform} independently to the sample points $Z_i$, we see that it is enough to consider the case in which every marginal of $\mu$ is uniform.

\smallskip

{\em Step 1. Dyadic boxes.}
Fix the resolution parameter
\begin{equation}	\label{eq: L}
	L=\lceil \log_2 (2n) \rceil.
\end{equation}  
A dyadic interval is an interval in [0,1] of the form 
$$
\big[ j2^{-\ell}, (j+1)2^{-\ell} \big]
\quad \text{for some } \ell \in \{0,1,\ldots,L-1\}, \; j \in \{0,1,\ldots,2^\ell-1\}.
$$
A dyadic box is a product of $d$ dyadic intervals. 

Almost surely no sample point $Z_i$ lies on the boundary of any dyadic box. Each point belongs to at most $L^d$ dyadic boxes, so the traces
\[
A_B \coloneqq \{i\in[N]:\; Z_i\in B\},\qquad B \text{ is a dyadic box},
\]
form an $L^d$-sparse family on $[N]$.

Apply Theorem~\ref{thm_sparse} to this family and to $S$.  We obtain $Q$ satisfying \eqref{eq:prop-overlap} and the discrepancy bound
\begin{equation}\label{eq:dyadic-disc}
	\disc(B) \le C\frac{L^{d/2}}{n} \left(\frac{k}{\delta}+\log(2n)\right)
	\quad \text{for any dyadic box } B,
\end{equation}
where 
\begin{equation}	\label{eq: discrepancy def}
	\disc(B) \coloneqq \abs*{ \frac1n\sum_{i\in Q}\one_{\{Z_i\in B\}} -\frac1N\sum_{i=1}^N\one_{\{Z_i\in B\}} }.
\end{equation}

\smallskip

{\em Step 2. Lattice boxes.}
A lattice interval is an interval of the type
$$
\big[ 0, j 2^{-L+1} \big]
\quad \text{for some } j \in \{0,1,\ldots,2^{L-1}\}.
$$
A lattice box is the product of $d$ lattice intervals.

Any lattice interval can be partitioned into at most $L$ dyadic intervals. Thus, any lattice box can be partitioned into at most $L^d$ dyadic boxes. Therefore, summing $L^d$ bounds \eqref{eq:dyadic-disc} by triangle inequality, we conclude
\begin{equation}	\label{eq: discrepancy lattice box}
\disc(B) 
\le C\frac{L^{3d/2}}{n} \left(\frac{k}{\delta}+\log(2n)\right)
	\quad \text{for any lattice box } B.
\end{equation}

\smallskip

{\em Step 3. Slices.}
We are about to approximate a general box by a dyadic box, with the difference between the two covered by a union of $d$ thin slices. To control the approximation error, we will need to bound the number of points $Z_i$ that can fall into a thin slice. Let's do that first.

A slice is the product of one interval of the type
\begin{equation}	\label{eq: slice}
	\big[ j2^{-L+1}, (j+1)2^{-L+1} \big]
	\quad \text{for some } j \in \{0,1,\ldots,2^{L-1}-1\},
\end{equation}
and $d-1$ intervals $[0,1]$; the interval can appear as any of the $d$ factors.

Since the sample points $Z_i$ have uniform marginals on $[0,1]$, the number of points falling into any slice $S$, 
$$
\sum_{i=1}^N \one_{\{Z_i \in S\}},
$$ 
has binomial distribution $\Bin(N, 2^{-L+1})$. Using the standard exponential moment method one can check that the expected maximum of $m$ random variables with distribution $\Bin(n,p)$ is bounded by $C(np + \log m)$; see e.g. \cite[Second proof of Proposition~2.7.6]{vershynin2026high}.
In our case, there are a total of $d2^{L-1}$ slices, so 
\begin{equation}	\label{eq: slice count}
	\E \max_{S: \text{ slice}} \sum_{i=1}^N \one_{\{Z_i \in S\}}
	\le C \Big( N2^{-L} + \log(d2^L) \Big).
\end{equation}

\medskip

{\em Step 4. Approximation.}
Any interval $B = [0,t] \subset [0,1]$ can be approximated from below by some lattice interval $B' \subset B$ by rounding $t$ down to the nearest multiple of $2^{-L+1}$; this way $B \setminus B'$ lies in some dyadic interval of the type \eqref{eq: slice}. Therefore, a general box 
$$
B = [0,t_1] \times \cdots \times [0,t_d]
$$
can be approximated from below by a lattice box $B' \subset B$ so that $B \setminus B'$ lies in a union of $d$ slices. 

By definition of discrepancy \eqref{eq: discrepancy def} and triangle inequality, we have 
\begin{equation}	\label{eq: discrepancy decomposed}
	\disc(B) \le \disc(B') + \disc(B \setminus B'), 
\end{equation}
and
\begin{align*}
	\disc(B \setminus B')
	&\le \frac1n\sum_{i\in Q}\one_{\{Z_i\in B \setminus B'\}} + \frac1N\sum_{i=1}^N\one_{\{Z_i\in B \setminus B'\}} \\
	&\le d \max_{S: \text{ slice}} \left( \frac1n\sum_{i\in Q}\one_{\{Z_i\in S\}} + \frac1N\sum_{i=1}^N\one_{\{Z_i\in S\}} \right)
		\quad \text{(since $B \setminus B'$ is covered by $d$ slices)} \\
	&\le d \max_{S: \text{ slice}} \left( \frac2N\sum_{i=1}^N\one_{\{Z_i\in S\}} + \disc(S) \right)
		\quad \text{(by definition of discrepancy)}.
\end{align*}
Putting this into \eqref{eq: discrepancy decomposed} and taking maximum over all anchored boxes, and then taking expectation, we obtain:
$$
\E \max_{B: \text{ box}} \disc(B) 
\le \underbrace{\E \max_{B': \text{ lattice box}} \disc(B')}_{E_1}
	+ \underbrace{\frac{2d}{N} \E \max_{S: \text{ slice}} \sum_{i=1}^N\one_{\{Z_i\in S\}}}_{E_2}
	+ \underbrace{d \E \max_{S: \text{ slice}} \disc(S)}_{E_3}.
$$
We can use \eqref{eq: discrepancy lattice box} to bound $E_1$, \eqref{eq: slice count} to bound $E_2$, and, since a slice is a dyadic box, \eqref{eq:dyadic-disc} can be used to bound $E_3$. 

Recalling the choice of the resolution parameter $L = \lceil \log_2 (2n) \rceil$ we made in \eqref{eq: L} and taking into account that $n>1$ (otherwise the conclusion is trivial), one can check that the bound for $E_1$ dominates both bounds for $E_2$ and $E_3$ (up to a factor $C^d)$, and it gives the right hand side of \eqref{eq:prop-discrepancy}.
\end{proof}

\bigskip

\begin{proof}[Proof of Theorem~\ref{thm_main}]
As we mentioned in the beginning of this section, we can assume that the measure $\mu$ is supported on $[0,1]^d$. Set
\[
\delta:=\frac{m}{3n}.
\]
If $\delta\le n^{-1/2}$, leave the sample unchanged.  Then \eqref{eq:VC} with $N=n$ gives
\[
\E \sup_B \abs*{\mu_n^X(B)-\mu(B)}
\le C\sqrt{\frac dn}
\le \frac{C\sqrt{d}}{\delta n}
\le \frac{3C\sqrt{d}}{m}.
\]
So, in this case we obtain a stronger conclusion than claimed in Theorem~\ref{thm_main}.

Assume now that $\delta>n^{-1/2}$.  Let
\begin{equation}	\label{eq: kN}
	k \coloneqq \lceil \log_2 (2n) \rceil,
	\qquad N \coloneqq n2^k.
\end{equation}
Let $Z_1,\ldots,Z_N$ be iid random variables with law $\mu$. Realize $X_1, \ldots, X_n$ as a random subsample of $Z_i$, i.e. set 
\begin{equation}    \label{eq: X-sample}
\{X_1,\ldots,X_n\} \coloneqq \{ Z_i: \; i \in S \},
\end{equation}
where $S\subset[N]$ is a uniformly random $n$-element subset independent of the full sample. 

Apply Proposition~\ref{prop_sparse_2} and obtain an $n$-element set $Q$. Set\footnote{Formally, we fix the canonical order of the $X$-sample in \eqref{eq: X-sample} and choose the order in the $Y$-sample \eqref{eq: Y-sample} to ensure that the overlap between $S$ and $Q$ translates to overlap between the samples. For example, if $S=\{5,7,8\}$ and $Q=\{7,8,9\}$, then we set $(X_1,X_2,X_3)=(Z_5,Z_7,Z_8)$ and $(Y_1,Y_2,Y_3)=(Z_9,Z_7,Z_8)$.}
\begin{equation}    \label{eq: Y-sample}
\{Y_1,\ldots,Y_n\} \coloneqq \{ Z_i: \; i \in Q \}.
\end{equation}
By \eqref{eq:prop-overlap} and $\delta>n^{-1/2}$,
\[
\E\#\{i:Y_i\ne X_i\}
= \E \abs{S \setminus Q} = 
n - \E \abs{S \cap Q} \le n\delta+2\sqrt n
\le 3n\delta=m.
\]
Finally, by the triangle inequality, Proposition~\ref{prop_sparse_2}, and \eqref{eq:VC} for the full sample,
\begin{align*}
	\E	\sup_B	\abs*{\mu_n^Y(B)-\mu(B)}
	&\le \E	\sup_B	\abs*{\mu_n^Y(B)-\mu_N^Z(B)}
		+ \E	\sup_B	\abs*{\mu_N^Z(B)-\mu(B)} \\
	&\le \frac{C^d \log^{3d/2}(2n)}{n}
		\left(\frac{k}{\delta}+\log(2n)\right) + \sqrt{\frac dN}.
\end{align*}
Due to the choice of $k$ and $N$ in \eqref{eq: kN}, we have $N \ge n^2$, and thus we conclude that 
$$
\E	\sup_B	\abs*{\mu_n^Y(B)-\mu(B)}
\le \frac{C_1^d \log^{3d/2+1}(2n)}{m}
$$
by adjusting an absolute constant $C_1$. 
\end{proof}

\bibliographystyle{plain}
\bibliography{ref}

\end{document}